\documentclass{amsart}
	\usepackage{aliascnt}
	\usepackage[colorlinks=true, linkcolor=black, citecolor=black, urlcolor=black]{hyperref}

	\newaliascnt{lemma}{thm}
	\newtheorem{lemma}[lemma]{Lemma}  
	\aliascntresetthe{lemma}

	\newaliascnt{prop}{thm}
	\newtheorem{prop}[prop]{Proposition} 
	\aliascntresetthe{prop}

	\newaliascnt{cor}{thm}
	\newtheorem{cor}[cor]{Corollary}
	\aliascntresetthe{cor}

	\theoremstyle{remark}

	\newaliascnt{rem}{thm}
	\newtheorem{rem}[rem]{Remark}
	\aliascntresetthe{rem}

	\theoremstyle{definition}

	\newaliascnt{exm}{thm}
	\newtheorem{exm}[exm]{Example}
	\aliascntresetthe{exm}

	\newaliascnt{notn}{thm}
	
	\aliascntresetthe{notn}

	\newaliascnt{defn}{thm}
	\newtheorem{defn}[defn]{Definition}
	\aliascntresetthe{defn}

	\usepackage{layout}
	
	\newcommand{\C}{\mathbb{C}}
	\newcommand{\K}{\mathbb{K}}
	\newcommand{\id}{\operatorname{id}}
	\newcommand{\sect}[1]{\operatorname{\Gamma}\left(#1\right)}
	\newcommand{\Hom}{\operatorname{Hom}}
	\newcommand{\Cp}[1]{\operatorname{C}^p\left(#1\right)}

\begin{document}

\title{On noncommutative equivariant bundles}
\author{Francesco D'Andrea}
\author{Alessandro De Paris}
\date{}
\address{Dipartimento di Matematica e Applicazioni ``Renato Caccioppoli'',\newline\indent Universit\`a di Napoli Fe\-de\-ri\-co II (Italy)}
\email{francesco.dandrea@unina.it, deparis@unina.it}

\begin{abstract}
We discuss a possible noncommutative generalization of the notion of an equivariant vector bundle. Let $A$ be a $\K$-algebra, $M$ a left $A$-module, $H$ a Hopf $\K$-algebra, $\delta:A\to H\otimes A:=H\otimes_\K A$ an algebra coaction, and let $(H\otimes A)_\delta$ denote $H\otimes A$ with the right $A$-module structure induced by~$\delta$. The usual definitions of an equivariant vector bundle naturally lead, in the context of $\K$-algebras, to an $(H\otimes A)$-module homomorphism \[\Theta:H\otimes M\to (H\otimes A)_\delta\otimes_AM\] that fulfills some appropriate conditions. On the other hand, sometimes an $(A,H)$-Hopf module is considered instead, for the same purpose. When $\Theta$ is invertible, as is always the case when $H$ is commutative, the two descriptions are equivalent.  We point out that the two notions differ in general, by giving an example of a noncommutative Hopf algebra $H$ for which there exists such a $\Theta$ that is not invertible and a left-right $(A,H)$-Hopf module whose corresponding homomorphism $M\otimes H\to (A\otimes H)_\delta\otimes_AM$ is not an isomorphism.
\end{abstract}

\maketitle



\section{Discussion}\label{Disc}
Here we discuss how an equivariant action ought to be generalized when Lie groups are replaced by Hopf algebras, actions on manifolds by coactions on algebras, and vector bundles by modules. This question came to our attention while we were reading \cite{LS}, and after a not-so-quick look at the literature we found a specific discussion only in \cite[Sect.~4]{S}. From both \cite[Subs.~3.1]{LS} and \cite[Subs.~4.4.1,~4.5.6]{S}, the reader might be induced to believe that a general (noncommutative) algebraic generalization of an equivariant bundle should consist of a module coaction over an algebra coaction of a Hopf algebra, that is, of a (relative) Hopf module (also called a Doi-Hopf module).

\subsection{Relative Hopf Modules}\label{RHM}

Let $\K$ be a field and $H$ a $\K$-bialgebra with comultiplication $\Delta:H\to H\otimes H$ (\footnote{Tensor products of modules without indication of the base ring are understood over $\K$; algebras are assumed to be associative and unital, and modules over them are unital; coalgebras are assumed to be coassociative and counital, and comodules over them are counital.}) and counit $\varepsilon:H\to\K$. In this work we consider left coactions and left modules (in \cite[4.4.1]{S} right coactions are considered instead). Thus, let $A$ be a left $H$-comodule algebra, with coaction \[\delta:A\to H\otimes A\;,\] that is, $\delta$ is a coalgebra left coaction on the vector space $A$ and a $\K$-algebra homomorphism. We also fix a left $A$-module  $M$. By saying that
\[
\overline{\delta}:M\to H\otimes M
\]
is a \emph{module coaction over $\delta$}, we mean that $\overline{\delta}(am)=\delta(a)\overline{\delta}(m)$ for all $a\in A,m\in M$, and that $\overline{\delta}$ is a coaction of $H$ on the vector space $M$.  We can equivalently say that $M$ is a \emph{left-left relative $(A,H)$-Hopf module}, following \cite[4.4.1]{S}.

Hopf modules in their simplest form, that is, $A=H$ with $\delta$ being the comultiplication, are treated in \cite{Sw}. The generalization to the relative ones was introduced and studied in a slightly less general setting (see \cite{T2}), under the hypothesis that $A$ is a coideal subalgebra, that is, $A$ is a subalgebra of $H$ such that $\Delta$ can be restricted to $\delta$. The present notion was introduced by Y.~Doi in \cite{D} (but here we prefer to refer to Doi's $(A,B)$--Hopf modules as right-right relative $(B,A)$--Hopf modules). In \cite[4.4.1]{S}, as well as in other papers (see, e.g., \cite[p.~111]{M} or \cite[Def.~2.1]{Sc}), left-right structures are considered.

Although Hopf modules are a natural notion for the description of noncommutative equivariant bundles, elementary considerations indicate indicate that this notion may not work in some pathological cases, at the basic level of generality where groups are replaced by Hopf algebras (cf.\ \cite[3.2]{S}). Let us now explain to some extent what these considerations are.

\subsection{Bundle Morphisms and Module Homomorphisms}

Let $\overline{f}:E\to E'$ be a morphism of the vector bundles $\pi:E\to X$ and $\pi':E'\to X'$, over a base morphism $f:X\to X'$ (that is, $\pi'\circ \overline{f}=f\circ\pi$ and the induced maps on the fibers are linear). 
Suppose that geometric structures on $X$ and $X'$ can be suitably encoded by algebras $A$ and $A'$. 
For instance, if $X$ and $X'$ are $\operatorname{C}^p$-manifolds, then $A=\Cp{X}$ and $A'=\Cp{X'}$; if $X$ and $X'$ are algebraic affine varieties, then $A=\mathcal{O}(X)$ and $A'=\mathcal{O}\left(X'\right)$. 
In these examples, the bundle structures can be encoded by the modules $M=\sect\pi$ and $M'=\sect{\pi'}$ of structure preserving ($\operatorname{C}^p$ or regular algebraic) global sections; the algebraic counterpart of $f$ is an algebra homomorphism $\varphi:A'\to A$ and (a structure preserving) $\overline{f}$ corresponds to an $A$-module homomorphism
\[
\overline\psi:M\to A\otimes_{A'}M'
\]
(\footnote{We explicitly mention that $A\otimes_{A'}M'$ is the module of sections of the pull-back $f^\ast\pi'$ and that taking global sections gives a \emph{covariant} functor $\operatorname{\Gamma}$. 
Even when sheaves are needed (e.g., for quasi-projective algebraic varieties), morphisms of vector bundles with the same base manifold give rise to morphisms of the corresponding sheaves of sections in a covariant way. 
For sheaves over schemes, a contravariant correspondence (basically, dual to the former) may also be considered (see \cite[Chap. II, Exer.~5.7]{H}), but we will not adopt that viewpoint. 
Note also that in \cite[p.180, Definition]{H}, the tangent sheaf is the sheaf of sections of the tangent bundle, so that they do not correspond to each other via the contravariant correspondence $\mathbf{V}$ introduced in the mentioned Exercise. 
It is worth remarking that the algebraic description of vector bundles by modules (or more generally, by sheaves) is appropriate in the situations when the role of the total spaces can be encapsulated in the properties of vector bundle morphisms. 
In applications for which some analysis on the total spaces is needed, the algebraic counterpart of bundles may become more complicated (cf.\ \cite[1.1.13]{DV}).}). 

Since $A\otimes_{A'}M'$ is the module obtained by extension of scalars via $\varphi$, if $\overline\psi$ is invertible, then the inverse homomorphism ${\overline\psi\,}^{-1}:A\otimes_{A'}M'\to M$ naturally corresponds to an $A'$-module homomorphism of $M'$ to the $A'$-module obtained from $M$ by restriction of scalars via~$\varphi$ (\footnote{It suffices to compose it with the natural homomorphism $m'\mapsto 1\otimes m'$, and it is just the assertion that extension of scalars and restriction of scalars are adjoint functors.}). We can equivalently say that we have a module homomorphism
\[
\overline\varphi:M'\to M
\]
\emph{over $\varphi:A'\to A$} (the \emph{base} algebra homomorphism), by meaning with this that $\overline\varphi$ is additive and
\[
\overline\varphi(am)=\varphi(a)\overline\varphi(m)\;,\qquad\forall a\in A,m\in M\;.
\]

Note that this simpler algebraic counterpart of $\overline f$ can always be employed when $\overline f$ is given by the action of a group element (and under the assumption that global sections suffice for an equivalent algebraic description). 
Indeed, in naive terms, an action of a group $G$ on $X$ consists of a family $\left\{\alpha_g\right\}_{g\in G}$ of transformations of $X$ into itself such that $\alpha_1=\id_X$ and $\alpha_{gg'}=\alpha_g\circ\alpha_{g'}$; an equivariant action on $\pi$ consists of a family $\left\{\overline\alpha_g\right\}_{g\in G}$ of morphisms of $\pi$ into itself such that $\overline\alpha_g$ is a morphism over $\alpha_g$ for each $g$, $\overline\alpha_1=\id_E$ and $\overline\alpha_{gg'}=\overline\alpha_g\circ\overline\alpha_{g'}$. 
Then $\overline\alpha_g$ induces an isomorphism simply because it is a vector bundle isomorphism (with $\overline\alpha_{g^{-1}}$ as its inverse morphism).

\subsection{Families of Bundle Morphisms}

In the situation we have just described, usually $G$ comes endowed with a geometric structure of the same kind as that on $X$, and the action is regular with respect to these structures and the vector bundle structure (\footnote{We also mention that results in Chapter~5 of the Grothendieck's \textit{T\^ohoku} paper encompass nonregular actions, too.}). 
The regularity hypothesis on the base is easily encoded by requiring that the family $\left\{\alpha_g\right\}_{g\in G}$ comes from a morphism $\alpha:G\times X\to X$, simply by setting $\alpha_g:=\alpha\circ\left(g,\id_X\right)$ for all $g$. 
To encode the equivariant action $\overline\alpha$ on $\pi$, one has to consider a vector bundle morphism over $\alpha$ such that for each $g$ we have a morphism $\overline\alpha_g$ of $\pi$ into itself over $\alpha_g$. 
To this end, the domain of $\overline\alpha$ must be $p_2^\ast\pi$, with $p_2:G\times X\to X$ being the projection map, so that $\left(g,\id_X\right)$ can naturally lift to a morphism of $\pi$ into that domain, for all $g$. 
Hence equivariant actions are given by vector bundle morphisms $\overline\alpha$ of $p_2^\ast\pi$ into $\pi$ over $\alpha$. 
The same motivation holds, more generally, for the definition of regular families of vector bundle morphisms: they are given by vector bundle morphisms of $p_2^\ast\pi$ to $\pi'$ over a morphism $T\times X\to X'$, with $T$ being a space of parameters.

In the above situation, to exploit the regularity of the action, one has to work with the map $\overline{\alpha}$ rather than with the family $\left\{\overline\alpha_g\right\}_{g\in G}$. This is a basic level at which the algebraic formulation we are discussing comes into play, and of course that formulation becomes fundamental in the noncommutative context. 
In general, even in the classic context of $\operatorname{C}^p$-manifolds, formulations like these deserve some attention (cf.\ \cite{N} and \cite{DV}).

In the context of affine varieties (and affine group varieties) over $\K$, the algebraic counterpart of $\alpha$ is a $\K$-algebra homomorphism
\[
\delta:A\to\mathcal{O}(G\times X)=H\otimes A\;,
\]
with $H:=\mathcal{O}(G)$ (\footnote{Strictly speaking, the identification $\mathcal{O}(G\times X)=H\otimes A$ is allowed when $\K$ is algebraically closed. It also works with no trouble for every $\K$, provided that varieties are considered as instances of schemes over $\K$.

For compact group actions on homogeneous spaces, there are canonical dense subalgebras $\mathcal{O}(X)$ and $\mathcal{O}(G)$ of $C(X)$ and $C(G)$ such that the coaction $C(X)\to C(G\times X)$ restricted to $\mathcal{O}(X)$ maps to the algebraic tensor product $\mathcal{O}(G)\otimes\mathcal{O}(X)$, i.e., becomes an algebraic coaction of the Hopf algebra $\mathcal{O(G)}$ on $\mathcal{O}(X)$. 
This is true even for compact quantum groups. 
See \cite{P} for details.

In the context of $\operatorname{C}^p$-manifolds, we also have an algebra homomorphism $\Cp X\to\Cp{M\times X}$, but for the description of $\Cp{M\times X}$, the ordinary tensor product does not work in general (see \cite[10.3]{N} and \cite[0.2.24]{DV}). 
A $\operatorname{C}^p$-tensor product and a related algebraic operation on modules for the description of pull-back bundles may easily be introduced (in other words, the whole situation may be described in a simple way in the context of monoidal and fibered categories of an algebraic kind).

For the sake of brevity, in the rest of the discussion we shall restrict ourselves to affine algebraic varieties as a guiding example.}). 
The algebraic counterpart of $\overline{\alpha}$, at least as an instance of the more general notion of a family of vector bundle morphisms, is given by an $(H\otimes A)$-module homomorphism from \[\sect{p_2^\ast\pi}=\left(H\otimes A\right)\otimes_AM=H\otimes M\] to \[\sect{\alpha^\ast\pi}=\left(H\otimes A\right)_\delta\otimes_AM\;,\] where $(H\otimes A)_\delta$ indicates that $H\otimes A$ is endowed with the $A$-module structure induced via $\delta$ (whereas, in the description of $\sect{p_2^\ast\pi}$, $H\otimes A$ is understood with the standard $A$-module structure induced by the second factor).

\subsection{The Isomorphism Hypothesis}

No surprise that the reasonable basic description we have outlined above can soundly be linked to the literature. Indeed, it is basically an instance in the affine (and ``commutative'') case of the descriptions that can be found, e.g., in \cite[0.2]{BL}, \cite[Def.~2.1]{K}, \cite[4.5.4]{S}, \cite[Chap.~1, Sec.~3, Def.~1.6]{MFK} (\footnote{In the latter reference one finds another friendly justification for the description we are dealing with, in the more general context of sheaves (though restricted to invertible sheaves, which correspond to line bundles).}). To be precise, there is a small but important difference. Indeed, the description that one gets in the affine (and commutative) case from the cited references consists of a module \emph{isomorphism}. Moreover, in our notation, that isomorphism goes from $\sect{\alpha^\ast\pi}$ to $\sect{p_2^\ast\pi}$, whereas we introduced a module homomorphism from $\sect{p_2^\ast\pi}$ to $\sect{\alpha^\ast\pi}$. Of course, once one has recognized that the homomorphism is indeed an isomorphism, no substantial difference is in view (\footnote{Another technical difference that some reader might have noted is that in \cite[Chap.~1, Sec.~3, Def.~1.6]{MFK} (and in \cite[4.5.4]{S}) some natural isomorphisms are explicitly displayed in what is called the \emph{cocycle condition}, whereas they are understood in \cite[Def.~2.1]{K}. Under appropriate technical conventions, some natural isomorphisms can be omitted at all: cf.\ \cite[0.1.1]{DV}, at the beginning of p.~2. In terms of fibered categories, these conventions basically consist (at least in the present situation) in that a cleavage is chosen in course of the exposition, a bit like for Grothendieck universes, or also for `generic objects' in classical Algebraic Geometry. We shall explicitly explain these conventions at the beginning of the next section, and some fundamental conditions, such as \eqref{Eb1} and \eqref{Eb2}, will be written under them.}).

For affine varieties, to recognize that we are dealing in fact with an isomorphism is quite easy, since $\overline\alpha$ induces isomorphisms on the fibers (from that over $(g,a)$ to that over $g\cdot a=\alpha(g,a)$, for each $g,a$). For schemes, one can not work `pointwise': usual techniques lead to consider the map
\[
G\times X\;\overset{\operatorname{diag}\times\id_X}{\longrightarrow}\;G\times G\times X\;\overset{i\times\alpha}{\longrightarrow}\; G\times X
\]
with $\operatorname{diag}$ and $i$ being the diagonal and the inverse maps. Alternatively, one can follow a category-theoretic approach: see \cite[Prop.~3.49]{V}. Note also that for affine group schemes one has, in addition, that they must be reduced, at least when $\K$ has characteristic zero (see \cite[Lec.~25, Th.~1]{M2}), and the morphisms $\alpha_g$ must be isomorphisms even when $X$ is nonreduced. Hence, even a pointwise approach might suffice (\footnote{It may seem a bit odd that the redundant isomorphism hypothesis has been required in the definitions we mentioned. One reason might be that in the context of works such as \cite{K} and \cite{MFK}, results such as \cite[Lec.~25, Th.~1]{M2} may have been considered as granted (note also that in \cite{K} there is a standing assumption that the ground field has characteristic zero). Hence the fact that the homomorphism involved is in fact an isomorphism might have been considered quite intuitive, if not obvious, and to put an explicit remark would have been distracting from the main focus. In \cite[0.2]{BL} they deal with topological spaces and groups, so the assumption that the considered map is an isomorphism is even more reasonable.}).

In any case, we also include in this paper the result that when $H$ is commutative we always have an isomorphism, as a consequence of (the algebraic counterparts of) the conditions that define an action: see \autoref{Comm} (\footnote{In the popular, but reliable, Wikipedia website one also finds a webpage on equivariant sheaves (at the time of writing it is \url{https://en.wikipedia.org/w/index.php?title=Equivariant_sheaf&oldid=835164209}). Even there, the isomorphism hypothesis is assumed in the definition (together with the cocycle condition). It is also noted that the action condition about the identity is a consequence, but the remark that the isomorphism condition is a consequence of the two action conditions is missing.}). We give a short direct algebraic proof, which may be useful for comparison with the noncommutative situation. When $A$ is commutative as well, \autoref{Comm} becomes an instance of \cite[Prop.~3.49]{V}. We mention that in \cite[Def.~3.46]{V} one finds a very general notion of an equivariant object, convincingly placed on the ground of fibered categories, which is also recalled in \cite[4.1]{S}. This notion can encompass also nonregular actions, such as those considered in the \textit{T\^ohoku} paper, and for group schemes gives the (regular) actions as defined in \cite[Chap.~1, Sec.~3, Def.~1.6]{MFK}.

\subsection{Conclusive Statements}

We have just outlined the following facts (some of which we are going to prove in detail in the next section):
\begin{itemize}
\item In the `commutative situation', an equivariant bundle corresponds to a module homomorphism
\[
\Theta:\sect{p_2^\ast\pi}=H\otimes M\longrightarrow\left(H\otimes A\right)_\delta\otimes_AM=\sect{\alpha^\ast\pi}
\]
(or, more generaly, to an analogous sheaf homomorphism) that must fulfill appropriate counterparts of the conditions that define an action.
\item These conditions imply that $\Theta$ must be an isomorphism (even when $A$, but not $H$, is not commutative).
\item By adjointness of extension and restriction of scalars, $\Theta^{-1}$ corresponds to a homomorphism $\overline\delta:M\to H\otimes M$ over $\delta$.
\item Again by the action conditions, $\overline\delta$ defines a relative Hopf module.
\end{itemize}

What we argue in this paper is that for some noncommutative Hopf algebras, contrary to the commutative case, the two (counterpart of) action conditions do not imply that an homomorphism \[\Theta:H\otimes M\to (H\otimes A)_\delta\otimes_AM\] must be an isomorphism. To this end, in \autoref{Ex} we shall use one of the simplest among the Hopf algebras whose antipode was shown to be not bijective in \cite{T1}, and exhibit a map $\Theta$ that is not an isomorphism. In this case, the simplified description given by a relative $(A,H)$--Hopf module does not apply. Moreover, in \autoref{ExHopf} we show that there exists a relative $(A,H)$--Hopf module that comes from no invertible map $\Theta$ as above.

Let us mention that a different simplified description can still be considered, again because of adjointness of extension and restriction of scalars. Namely, to assign $\Theta$ is the same as to assign the left $A$-module homomorphism
\[
\theta:M\to (H\otimes A)_\delta\otimes_AM\;,\qquad m\mapsto\Theta(1\otimes m)
\]
where the $A$-module structure on the target is induced by the second factor in $H\otimes A$, (\footnote{A similar option for coherent sheaves, in the situation of \cite[Chap.~1, Sec.~3, Def.~1.6]{MFK}, is to consider a morphism $L\to {p_2}_\ast\alpha^\ast L$, with $p_2$ and $\alpha$ as before.}).

Finally, we remark that the module considered in \autoref{Ex} is free (of rank two) and the base algebra is noncommutative. Hence, to view that example as an exotic kind of noncommutative equivariant vector bundle (trivial, of rank two) may be reasonable. More generally, at least at the algebraic level, it is not unreasonable to view projective (maybe also finitely generated) modules over noncommutative rings as noncommutative vector bundles, because of the Swan's theorem: \cite[Subs.~3.1]{LS} seems to adopt this viewpoint. In this frame, we would have that the `right definition' of a noncommutative equivariant vector bundle is given by the homomorphism $\Theta$ (or $\theta$), provided that $M$ is (at least) a projective module.
In the final section we present some related problems that arise in the context of fibered categories, mainly in connection with the results of \cite[Sect.~3.8]{V}.

\subsection{Acknowledgements}

We are grateful to Tomasz Brzezi\'nski and Zoran \v Skoda for providing us with some feedbacks on an early draft of the present work.

\section{An exotic noncommutative equivariant bundle}\label{Enb}

\subsection{Basic results and conventions}

For the reader convenience, we explicitly recall some elementary results and stipulate some conventions about tensor products and extension of scalars. To this end, let us consider a field $\K$, a $\K$-algebra $A$, a left $A$-module $M$ and a right $A$-module $M'$ (as anticipated in the previous section, they are all assumed to be unital).

We assume no fixed general construction of tensor products: $M'\otimes_AM$ can be any $\K$-vector space together with an $A$-balanced map $\beta:M'\times M\to M'\otimes_AM$, $\beta\left(m',m\right)=:m'\otimes m$ (\footnote{By saying that $\beta$ is \emph{$A$-balanced} we mean that it is $\K$-bilinear and $m'a\otimes m=m'\otimes am$, for all $m\in M$, $m'\in M'$, $a\in A$.}) that satisfies the universal property: for every $\K$-vector space $V$ and every $A$-balanced map $b:M'\times M\to V$ there exists a unique homomorphism $\overline b:M'\otimes_A M\to V$ of $\K$-vector spaces such that $b=\overline b\circ\beta$ (cf.\ \cite[p.~25, Rem.~iii]{AM}). Sometimes a particular choice of a tensor product will be convenient, and in this case it will be explicitly indicated. Tensor products of modules without indication of the base ring will be understood as tensor products of $\K$-vector spaces (sometimes equipped with module structures inherited from some additional module structures on the factors). From the universal property readily follows that for every given left $A$-module homomorphism $f:M_0\to M_1$ and right $A$-module homomorphism $f':M'_0\to M'_1$, there exists exactly one $\K$-vector space homomorphism \[f'\otimes_A f:M'_0\otimes_A M_0\to M'_1\otimes_A M_1\] such that $\left(f'\otimes_A f\right)\left(m\otimes m'\right)=f(m)\otimes f'\left(m'\right)$ for all $m\in M_0$, $m'\in M'_0$. We also recall that when $M$ and $M'$ are $\K$-algebras, $M'\otimes M$ is also a $\K$-algebra with the multiplication being the only one such that $\left(m'_1\otimes m_1\right)\left(m'_2\otimes m_2\right)=\left(m'_1m'_2\otimes m_1m_2\right)$ (\footnote{This holds, more generally, for $M'\otimes_RM$ when $R$ is a commutative ring and $M,M'$ are $R$-algebras.}).

We shall use the notation $\varphi^\ast$ for extension of scalars of left modules via a $\K$-algebra homomorphism $\varphi:A\to B$: \[\varphi^\ast M:=B\otimes_A M\;,\] with $B$ considered as a right $A$-module via $\varphi$ ($ba:=b\varphi(a)$), and with the naturally induced left $B$-module structure
\[
bx:=\left(\mu_b\otimes_A\id_M\right)(x)\;,\quad\forall b\in B,x\in\varphi^\ast M\;,
\]
where $\mu_b:B\to B$ is the multiplication by $b$ on the left (in other words, the structure is the unique one such that $b\left(b'\otimes m\right)=bb'\otimes m$ for all $b,b'\in B$ and $m\in M$).

Let us recall the universal property of extension of scalars. We have a natural map $\nu:M\to\varphi^\ast M$, $\nu(m):= 1\otimes m$ for all $m\in M$, which is a left module homomorphism over $\varphi$ (that is, $\nu(am)=\varphi(a)\nu(m)$) and is universal in the following sense: for every given left $B$-module $N$ and left module homomorphism $\overline\varphi:M\to N$ over $\varphi$, there exists exactly one left $B$-module homomorphism $f:\varphi^\ast M\to N$ such that $f\circ\nu=\overline\varphi$. We say that \emph{$f$ and $\overline\varphi$ correspond} to each other \emph{via $\varphi$}. Note also that, according to our conventions, we can assume that $A\otimes_AM=M$, $a\otimes m=am$, and with this choice we have $\nu=\varphi\otimes_A\id_M$.

From the universal property easily follows that, given a left $A$-module homorphism $g:M_0\to M_1$, there is exactly one left $B$-module homomorphism \[\varphi^\ast g:\varphi^\ast M_0\to\varphi^\ast M_1\] such that $\varphi^\ast g\circ\nu_0=\nu_1\circ g$, with $\nu_0,\nu_1$ being the natural maps. The homomorphism $\varphi^\ast g$ is said to be obtained from $g$ by extension of scalars via $\varphi$ (we also have $
\varphi^\ast g=\id_B\otimes g:B\otimes_AM_0\to B\otimes_AM_1$).

\subsection{Standing notation}\label{Standing}

Let us introduce some notations that will be considered as fixed in the rest of the paper. Let $\K$ be a field and $H$ a $\K$-bialgebra, with comultiplication $\Delta:H\to H\otimes H$, counit $\varepsilon:H\to\K$, unit $\eta:\K\to H$ and multiplication $\mu:H\otimes H\to H$. Let $A$ be a $\K$-algebra, $M$ a left $A$-module and \[\delta:A\to H\otimes A\] an algebra left coaction of $H$ on $A$, i.e., $\delta$ is a $\K$-algebra homomorphism and we have
\begin{equation}\label{Coaction}
\left(\Delta\otimes\id_A\right)\circ\delta=\left(\id_H\otimes\delta\right)\circ\delta\;,\qquad\left(\varepsilon\otimes\id_A\right)\circ\delta=\id_A\;,
\end{equation}
under the assumptions $H\otimes(H\otimes A)=(H\otimes H)\otimes A=:H\otimes H\otimes A$ and $\K\otimes A=A$ (\footnote{Each tensor product comes implicitly equipped with a balanced map: by writing these equalities of modules we also consider as understood that the balanced maps are such that $(h_1\otimes h_2)\otimes a=h_1\otimes(h_2\otimes a)=:h_1\otimes h_2\otimes a$ and $\lambda\otimes a=\lambda a$ for all $h_1,h_2\in H$, $a\in A$ and $\lambda\in\K$.}). The choices of these tensor products, as well as the choice $\K\otimes M=M$, will be kept along the paper as well.

Let $(H\otimes A)_\delta$ denote the $\K$-algebra $H\otimes A$ considered as a right $A$-module by means of $\delta$. We consider a map \[\theta:M\to (H\otimes A)_\delta\otimes_AM\;,\] we assume that the codomain is equipped with the left $(H\otimes A)$--module structure determined by the condition
\[
\left(h\otimes a\right)\left(\left(h'\otimes a'\right)\otimes m\right)=\left(hh'\otimes aa'\right)\otimes m\;,\quad\forall h,h'\in H,\;a,a'\in A,\;m\in M\;,
\]
and that $\theta$ is a module homomorphism over the natural algebra homomorphism $\nu:A\to H\otimes A$, $\nu(a):=1\otimes a$ (that is, $\theta(am)=\nu(a)\theta(m)$). We also consider the corresponding left $(H\otimes A)$--module homomorphism \[\Theta:H\otimes M\to (H\otimes A)_\delta\otimes_AM\] determined by the condition \[\Theta(1\otimes m)=\theta(m)\;,\quad\forall m\in M\]
(we shall soon rewrite $\Theta$ in notation of extension of scalars).

We shall sometimes assume the sumless Sweedler notation $\delta(a)=:a_{(-1)}\otimes a_{(0)}$ and $\Delta(h)=:h_{(1)}\otimes h_{(2)}$. We shall also make use of more elaborated Sweedler-like sumless notations, like
\[
\theta(m)=:m_{(-1)}\otimes m_{(0)}\otimes m_{(1)}\;,
\]
which is to be understood, as usual, as an abbreviation for \[\sum_i\left(\sum_jm_{-1,i,j}\otimes m_{0,i,j}\right)\otimes m_{1,i}\;, \quad m_{-1,i,j}\in H,\;m_{0,i,j}\in A,\;m_{1,i}\in M\;.\]
The combined use of these notations requires some caution, especially in the case when $M=A$, because of potential ambiguities. We shall make use of it only in a few situations, where will be convenient and sufficiently safe.

\subsection{A noncommutative notion for equivariant vector bundles}

As we explained in \autoref{Disc}, at least when $H$ is a Hopf algebra and $M$ is projective and finitely generated, $\theta$ (or, equivalently, $\Theta$) can be considered as a noncommutative equivariant bundle, provided that some algebraic conditions, encoding the geometric action conditions, are satisfied (\footnote{To be precise, it is not $\theta$ (or $\Theta$) alone that works, because the understood balanced maps of the tensor products must also be taken into account.}). These conditions can be efficiently written by means of Sweedler-like notations:
\begin{multline}\label{SwEb1}
m_{(1)(-1)}\otimes m_{(-1)}m_{(1)(0)(-1)}\otimes m_{(0)}m_{(1)(0)(0)}\otimes m_{(1)(1)}\\
=m_{(-1)(1)}\otimes m_{(-1)(2)}\otimes m_{(0)}\otimes m_{(1)}
\end{multline}
in $\left(H\otimes H\otimes A\right)_\gamma\otimes_AM$, with $\gamma:=\left(\Delta\otimes\id_A\right)\circ\delta=\left(\id_H\otimes\delta\right)\circ\delta$, and
\begin{equation}\label{SwEb2}
\varepsilon\left(m_{(-1)}\right)m_{(0)}m_{(1)}=m\;.
\end{equation}

We also remind that the algebraic description of vector bundles by means of modules (or of sheaves; cf.\ \cite[Chap. II, Exer.~5.7]{H}) is appropriate when the role of the total spaces can be encapsulated in the properties of vector bundle morphisms (in other words when, in order to use vector bundles, to consider them as members of their category basically suffices). In the applications for which some analysis on the total spaces is needed, the algebraic counterpart of bundles gets more complicated (cf.\ \cite[1.1.13]{DV}).

\subsection{The conditions for noncommutative equivariant bundles as homomorphism identities}
Recall that we are keeping the tensor product choice $\K\otimes A=A$, $\K\otimes M=M$, and let us similarly assume $\K\otimes(H\otimes A)=H\otimes A$ (for some choice of $H\otimes A$). Hence $\eta\otimes\id_A:A\to H\otimes A$ is the map $a\mapsto 1\otimes a$ and we can also assume that
\begin{equation}\label{HM}
\left(\eta\otimes\id_A\right)^\ast M=H\otimes M
\end{equation}
(for some choice of $H\otimes M$) with
\[
\left(\eta\otimes\id_A\right)^\ast M=\left(H\otimes A\right)\otimes_{A}M\ni (h\otimes a)\otimes m=h\otimes am\in H\otimes M\;,
\]
so that the natural map $M\to\left(\eta\otimes\id_A\right)^\ast M=H\otimes M$ is $\eta\otimes\id_M$, $m\mapsto 1\otimes m$. With these assumptions, $\Theta$ is a left $(H\otimes A)$--module homomorphism \[\left(\eta\otimes\id_A\right)^\ast M\to\delta^\ast M\;,\] $\theta$ a left module homomorphism \[M\to\delta^\ast M\] over the algebra homomorphism $\eta\otimes\id_A$, and they correspond to each other via $\eta\otimes\id_A$ (\footnote{Using a notation $\varphi_\ast$ for restriction of scalars through an algebra homomorphism $\varphi:A\to B$, we can also consider $\theta$ as a left $A$-module homomorphism $M\to\left(\eta\otimes\id_A\right)_\ast\delta^\ast M$.}).
Since
\[
\left(\eta\otimes\id_{H\otimes A}\right)\circ\delta=\eta\otimes\delta=\left(\id_H\otimes\delta\right)\circ\left(\eta\otimes\id_A\right)\;,
\]
we can assume that extensions of scalars are chosen so that
\begin{equation}\label{C1}
\left(\eta\otimes\id_{H\otimes A}\right)^\ast\delta^\ast M=\left(\id_H\otimes\delta\right)^\ast\left(\eta\otimes\id_A\right)^\ast M.
\end{equation}
(By writing this, we also tacitly assume that the natural maps of $M$ into that $(H\otimes H\otimes A)$--module, induced via the two compositions, are the same.)
Hence the composition
\[
\left(\id_H\otimes\delta\right)^\ast\Theta\circ\left(\eta\otimes\id_{H\otimes A}\right)^\ast\Theta
\]
makes sense. Moreover, we can choose
\begin{equation}\label{C2}
\left(\eta\otimes\id_{H\otimes A}\right)^\ast\left(\eta\otimes\id_A\right)^\ast M=\left(\Delta\otimes\id_A\right)^\ast\left(\eta\otimes\id_A\right)^\ast M\;,
\end{equation}
and taking into account the first identity in \eqref{Coaction}, also
\begin{equation}\label{C3}
\left(\id_H\otimes\delta\right)^\ast\delta^\ast M=\left(\Delta\otimes\id_A\right)^\ast\delta^\ast M\;.
\end{equation}
With these assumptions we can write the condition \eqref{SwEb1} as
\begin{equation}\label{Eb1}
\left(\id_H\otimes\delta\right)^\ast\Theta\circ\left(\eta\otimes\id_{H\otimes A}\right)^\ast\Theta=\left(\Delta\otimes\id_A\right)^\ast\Theta\;,
\end{equation}
with no involvement of natural isomorphisms such as those in, e.g., \cite[4.5.4]{S}, since they are actually identity maps because of the choices \eqref{C1}, \eqref{C2}, \eqref{C3}. Similarly, we can choose
\begin{equation}\label{C4}
\left(\varepsilon\otimes\id_A\right)^\ast\left(\eta\otimes\id_A\right)^\ast M=M\;(\footnote{Of course, by writing this we also understood that the natural homomorphism \[\left(\eta\otimes\id_A\right)^\ast M\to M\] maps $h\otimes m$ into $\varepsilon(h)m$. This assures, in particular, that the composition of the natural maps $M\to\left(\eta\otimes\id_A\right)^\ast M\to M$ is the natural map $M\to M$ that comes by extension of scalars of $M$ over $\id_A$, that is $\id_M$. A similar assumption for \eqref{HM} was explicitly written, whereas for \eqref{C1}, \eqref{C2} and \eqref{C3} was considered as understood as well.}),
\end{equation}
and, taking into account the second identity in \eqref{Coaction}, also
\begin{equation}\label{Cnew} \left(\varepsilon\otimes\id_A\right)^\ast\delta^\ast M=M\;.
\end{equation}
Then \eqref{SwEb2} is equivalent to
\begin{equation}\label{Eb2}
\left(\varepsilon\otimes\id_A\right)^\ast\Theta=\id_M\;.
\end{equation}
To recognize that \eqref{Eb1} is equivalent to \eqref{SwEb1} it may be helpful to rewrite it as
\begin{equation}\label{Tid}
\left(\id_H\otimes\delta\right)^\ast\Theta\circ\left(\id_H\otimes\Theta\right)=\left(\Delta\otimes\id_A\right)^\ast\Theta\;,
\end{equation}
under the choices
\[
\left(\eta\otimes\id_{H\otimes A}\right)^\ast\left(\eta\otimes\id_A\right)^\ast M=H\otimes\left(\eta\otimes\id_A\right)^\ast M\;,\quad\left(\eta\otimes\id_{H\otimes A}\right)^\ast\delta^\ast M=H\otimes\delta^\ast M\;.
\]

From now on in this paper, along with the standing notation $\K$, $A$, $M$, $H$, $\eta$, $\varepsilon$, $\Delta$, $\mu$, $\theta$, $\Theta$ introduced in \autoref{Standing},
\begin{itemize}
\item we shall always consider \eqref{HM} as implicitly assumed;
\item the condition \eqref{Eb1} will be always understood under the assumptions \eqref{C1}, \eqref{C2} and \eqref{C3};
\item the condition \eqref{Eb2} will be always understood under the assumptions \eqref{C4} and \eqref{Cnew}.
\end{itemize}

To summarize, homomorphisms $\Theta$ for which $\theta$ satisfies \eqref{SwEb1} and \eqref{SwEb2} (the main subject of the present work), are the $\Theta$s that satisfies \eqref{Eb1} and \eqref{Eb2}.

\subsection{Noncommutative equivariant bundles and Hopf modules}

When $\Theta$ is an isomorphism, $\Theta^{-1}$ corresponds via $\delta$ to a module homomorphism over $\delta$:
\[
\overline{\delta}:M\to\left(\eta\otimes\id_A\right)^\ast M\overset{\eqref{HM}}{=}H\otimes M\;.
\]
Below we shall explicitly show that \eqref{Eb1} and \eqref{Eb2} correspond to the coaction condition on $\overline\delta$. To this end, for the reader convenience we preliminarily state an elementary result.

\begin{prop}\label{Elext}
Let $\varphi:A\to B$, $\psi:B\to C$ be $\K$-algebra homomorphisms, $\overline\varphi:M\to N$ a left module homomorphism over $\varphi$, $\overline\psi:N\to P$ a left module homomorphism over $\psi$, $f:\varphi^\ast M\to N$ the left $B$-module homomorphism corresponding to $\overline\varphi$ via $\varphi$ and $g:\psi^\ast N\to P$ the left $C$-module homomorphism corresponding to $\overline\psi$ via $\psi$. Assuming \[\left(\psi\circ\varphi\right)^\ast M=\psi^\ast\varphi^\ast M\;,\] we have that $g\circ\psi^\ast f$ and $\overline\psi\circ\overline\varphi$ correspond to each other via $\psi\circ\varphi$.
\end{prop}
\begin{proof}
Assuming $A\otimes_AM=M$, $B\otimes_BN=N$, the natural homomorphisms $M\to\varphi^\ast M$ and $N\to\psi^\ast N$ are $\varphi\otimes_A\id_M$ and $\psi\otimes_B\id_N$, respectively. By definition of $f$ and $g$, we have
\[
\overline\varphi=f\circ\left(\varphi\otimes_A\id_M\right)\;,\qquad\overline\psi=g\circ\left(\psi\otimes_B\id_N\right)\;.
\]
Since we are assuming $\psi^\ast\varphi^\ast M=\left(\psi\circ\varphi\right)^\ast M$ and under the choice $B\otimes_B\varphi^\ast M=\varphi^\ast M$ as well, we can also write
\[
\left(\psi\otimes_B\id_{\varphi^\ast M}\right)\circ\left(\varphi\otimes_A\id_M\right)=\left(\psi\circ\varphi\right)\otimes_A\id_M\;.
\]
Moreover, $\psi^\ast f$ is the unique $C$-module homomorphism such that
\[
\psi^\ast f\circ\left(\psi\otimes_B\id_{\varphi^\ast M}\right)=\left(\psi\otimes_B\id_N\right)\circ f\;.
\]
Hence we have
\begin{multline*}
\overline\psi\circ\overline\varphi=g\circ\left(\psi\otimes_B\id_N\right)\circ f\circ\left(\varphi\otimes_A\id_M\right)\\
=g\circ\psi^\ast f\circ\left(\psi\otimes_B\id_{\varphi^\ast M}\right)\circ\left(\varphi\otimes_A\id_M\right)=g\circ\psi^\ast f\circ\left(\left(\psi\circ\varphi\right)\otimes_A\id_M\right)\;,
\end{multline*}
which precisely says that $g\circ\psi^\ast f$ and $\overline\psi\circ\overline\varphi$ correspond to each other via $\psi\circ\varphi$.
\end{proof}

\begin{prop}\label{Rho1}
If
\[
\rho:\delta^\ast M\to\left(\eta\otimes\id_A\right)^\ast M
\] 
is a left $(H\otimes A)$--module homomorphism and
\[
\overline{\delta}:M\to\left(\eta\otimes\id_A\right)^\ast M\overset{\eqref{HM}}{=}H\otimes M
\]
is the module homomorphism over $\delta$ corresponding to $\rho$ via $\delta$ then, under the assumptions \eqref{C1}, \eqref{C2}, \eqref{C3} and $(H\otimes H)\otimes M=H\otimes(H\otimes M)=:H\otimes H\otimes M$, we have
\begin{multline}\label{Lhs}
\left(\Delta\otimes\id_A\right)^\ast\rho=\left(\eta\otimes\id_{H\otimes A}\right)^\ast\rho\circ\left(\id_H\otimes\delta\right)^\ast\rho\\
\iff\left(\Delta\otimes\id_M\right)\circ\overline{\delta}=\left(\id_H\otimes\overline{\delta}\right)\circ\overline{\delta}\;.
\end{multline}
\end{prop}
\begin{proof}
Let
\[
g:\left(\id_H\otimes\delta\right)^\ast\left(H\otimes M\right)\to H\otimes H\otimes M
\]
be the $(H\otimes H\otimes A)$--module homomorphism corresponding to $\id_H\otimes\overline\delta$ via $\id_H\otimes\delta$. Following the standing assumption \eqref{HM}, the natural map $M\to\left(\eta\otimes\id_A\right)^\ast M$ is $\eta\otimes\id_M$  (under the other standing assumption $\K\otimes M=M$). Therefore the $(H\otimes A)$--module homomorphism $\left(\eta\otimes\id_A\right)^\ast M\to H\otimes M$ corresponding to $\eta\otimes\id_M$ via $\eta\otimes\id_A$ is the identity map. According to \autoref{Elext}, under the assumption
\[
\left(\eta\otimes\delta\right)^\ast M=\left(\id_H\otimes\delta\right)^\ast\left(\eta\otimes\id_A\right)^\ast M
\]
$g$ is also the $(H\otimes H\otimes A)$--module homomorphism corresponding to $\left(\id_H\otimes\overline\delta\right)\circ\left(\eta\otimes\id_M\right)=\eta\otimes\overline\delta$ via $\left(\id_H\otimes\delta\right)\circ\left(\eta\otimes\id_A\right)=\eta\otimes\delta$.

Exploiting \autoref{Elext} in a similar way for the composition $\left(\eta\otimes\id_{H\otimes M}\right)\circ\overline\delta=\eta\otimes\overline\delta$, and taking into account \eqref{C1}, we deduce that
\begin{equation}\label{g}
g=\left(\eta\otimes\id_{H\otimes A}\right)^\ast\rho\;.
\end{equation}
Note that $H\otimes H\otimes M$ is an extension of scalars $\left(\Delta\otimes\id_A\right)^\ast\left(H\otimes M\right)$ with natural map $h\otimes m\mapsto\Delta(h)\otimes m$, and since a change in the choice of $H\otimes H\otimes M$ of course does not affect the equality $\left(\Delta\otimes\id_M\right)\circ\overline{\delta}=\left(\id_H\otimes\overline{\delta}\right)\circ\overline{\delta}$, we can assume that
\[
H\otimes H\otimes M=\left(\Delta\otimes\id_A\right)^\ast\left(\eta\otimes\id_A\right)^\ast M\;.
\]
Hence the $\left(H\otimes H\otimes M\right)$--module homomorphism $\left(\Delta\otimes\id_A\right)^\ast\left(\eta\otimes\id_A\right)^\ast M\to H\otimes H\otimes M$ corresponding to $\Delta\otimes\id_M$ via $\Delta\otimes\id_A$, is the identity map. Exploiting \autoref{Elext} like before, we have that
\begin{itemize}
\item $\left(\Delta\otimes\id_A\right)^\ast\rho$ and $\left(\Delta\otimes\id_M\right)\circ\overline{\delta}$ correspond to each other via $\left(\Delta\otimes\id_A\right)\circ\delta$.
\end{itemize}
Similarly, taking into account \eqref{C2} and \eqref{g}, we have that
\begin{itemize}
\item$\left(\eta\otimes\id_{H\otimes A}\right)^\ast\rho\circ\left(\id_H\otimes\delta\right)^\ast\rho$ and $\left(\id_H\otimes\overline{\delta}\right)\circ\overline{\delta}$ correspond to\\ each other via $\left(\id_H\otimes\delta\right)\circ\delta$.
\end{itemize}
Taking into account \eqref{Coaction} and \eqref{C3}, we have \eqref{Lhs}.
\end{proof}

\begin{cor}\label{EbC1}
If $\Theta$ is an isomorphism then
\[
\eqref{Eb1}\iff\left(\Delta\otimes\id_M\right)\circ\overline{\delta}=\left(\id_H\otimes\overline{\delta}\right)\circ\overline{\delta}\;,
\]
under the assumptions \eqref{C1}, \eqref{C2}, \eqref{C3} and $(H\otimes H)\otimes M=H\otimes(H\otimes M)=:H\otimes H\otimes M$.
\end{cor}
\begin{proof}
Setting $\rho:=\Theta^{-1}$, we have that the left-hand statement in \eqref{Lhs} is clearly equivalent to \eqref{Eb1}.
\end{proof}

\begin{prop}\label{Rho2}
If
\[
\rho:\delta^\ast M\to\left(\eta\otimes\id_A\right)^\ast M
\] 
is a left $(H\otimes A)$--module homomorphism and
\[
\overline{\delta}:M\to\left(\eta\otimes\id_A\right)^\ast M\overset{\eqref{HM}}{=}H\otimes M
\]
is the module homomorphism over $\delta$ corresponding to $\rho$ via $\delta$ then, under the assumptions \eqref{C4} and \eqref{Cnew}, we have 
\begin{equation}\label{rho2}
\left(\varepsilon\otimes\id_A\right)^\ast\rho=\left(\varepsilon\otimes\id_M\right)\circ\overline{\delta}\;.
\end{equation}
\end{prop}
\begin{proof}
Taking into account \eqref{C4}, let $g:M\to M$ be the homomorphism corresponding to $\varepsilon\otimes\id_M$ via $\varepsilon\otimes\id_A$. By \autoref{Elext} we have that $g$ and $\left(\varepsilon\otimes\id_M\right)\circ\left(\eta\otimes\id_M\right)=\id_M$ correspond to each other via $\id_A$. Hence $g=\id_M$, because $\id_M$ obviously correspond to itself via $\id_A$.

Again by \autoref{Elext}, taking into account \eqref{Cnew}, we have that $\id_M\circ\left(\varepsilon\otimes\id_A\right)^\ast\rho$ and $\left(\varepsilon\otimes\id_M\right)\circ\overline{\delta}$ correspond to each other via $\left(\varepsilon\otimes\id_A\right)\circ\delta=\id_A$. This leads to \eqref{rho2}, because $\left(\varepsilon\otimes\id_M\right)\circ\overline{\delta}$ correspond to itself via $\id_A$.
\end{proof}

The following consequence is immediate.

\begin{cor}\label{EbC2}
If $\Theta$ is an isomorphism then
\[
\eqref{Eb2}\iff\left(\varepsilon\otimes\id_M\right)\circ\overline{\delta}=\id_M\;,
\]
under the assumptions \eqref{C4} and \eqref{Cnew}.
\end{cor}

According to Corollaries \ref{EbC1} and \ref{EbC2}, if $\Theta$ is invertible and satisfies \eqref{Eb1} and \eqref{Eb2}, then the homomorphism $\overline{\delta}$ corresponding to $\Theta^{-1}$ via $\delta$ is a coaction (over $\delta$). Thus we have a left-left relative $\left(A,H\right)$--Hopf module (according to the definition in \cite[4.4.1]{S}). Conversely, given a left-left relative $\left(A,H\right)$--Hopf module, that is, a coaction $\overline\delta:M\to H\otimes M$ over $\delta$, if the corresponding $(H\otimes A)$--module homomorphism
\[
\rho:\delta^\ast M\to H\otimes M
\]
is an isomorphism then, according to Corollaries \ref{EbC1} and~\ref{EbC2}, we get an isomorphism $\Theta:=\rho^{-1}$ that satisfies \eqref{Eb1} and \eqref{Eb2}. Moreover, we already pointed out that $\Theta$ satisfies \eqref{Eb1} and \eqref{Eb2} if and only if the corresponding homomorphism $\theta$ satisfies \eqref{SwEb1} and \eqref{SwEb2}.

In conclusion, the description of noncommutative equivariant bundles by means of (projective and finitely generated) relative Hopf modules agrees with the description given by a $\theta$ that satisfies \eqref{SwEb1} and \eqref{SwEb2}, when the corresponding $\Theta$ (that satisfies \eqref{Eb1} and \eqref{Eb2}) is an isomorphism.

\subsection{The case of commutative Hopf algebras}
Here we explicitly show that for commutative Hopf algebras $H$ (but still for arbitrary algebras $A$), a homomorphism $\Theta$ that satisfies \eqref{Eb1} and \eqref{Eb2} is always an isomorphism. In the following lemma we introduce an auxiliary map.

\begin{lemma}\label{Isigma}
When $H$ is a Hopf algebra with antipode $S:H\to H$, if we define
\[
\sigma:=\left(\mu\otimes\id_A\right)\circ\left(S\otimes\delta\right):H\otimes A\to H\otimes A\;,
\]
then we have
\begin{equation}\label{Inv}
\sigma\circ\delta=\eta\otimes\id_A\;,\qquad\sigma\circ\left(\eta\otimes\id_A\right)=\delta\;.
\end{equation}
\end{lemma}
\begin{proof}
We have
\begin{multline*}
\sigma\circ\delta=\left(\mu\otimes\id_A\right)\circ\left(S\otimes\id_H\otimes\id_A\right)\circ\left(\id_H\otimes\delta\right)\circ\delta\\
\overset{\eqref{Coaction}}{=}\left(\mu\otimes\id_A\right)\circ\left(S\otimes\id_H\otimes\id_A\right)\circ\left(\Delta\otimes\id_A\right)\circ\delta=\left(\left(\mu\circ\left(S\otimes\id_H\right)\circ\Delta\right)\otimes\id_A\right)\circ\delta\\
=\left(\left(\eta\circ\varepsilon\right)\otimes\id_A\right)\circ\delta=\left(\eta\otimes\id_A\right)\circ\left(\varepsilon\otimes\id_A\right)\circ\delta\overset{\eqref{Coaction}}{=}\eta\otimes\id_A\;.
\end{multline*}
Taking into account that $S\circ\eta=\eta$, using the assumption $\K\otimes\left(H\otimes A\right)=H\otimes A$ we have
\begin{multline*}
\sigma\circ\left(\eta\otimes\id_A\right)=\left(\mu\otimes\id_A\right)\circ\left(S\otimes\delta\right)\circ\left(\eta\otimes\id_A\right)=\left(\mu\otimes\id_A\right)\circ\left(\eta\otimes\delta\right)\\
=\left(\mu\otimes\id_A\right)\circ\left(\eta\otimes\id_H\otimes\id_A\right)\circ\left(\id_{\K}\otimes\delta\right)=\left(\left(\mu\circ\left(\eta\otimes\id_H\right)\right)\otimes\id_A\right)\circ\delta\\
=\left(\id_H\otimes\id_A\right)\circ\delta=\delta\;.
\end{multline*}
as required.

Alternatively, in (sumless) Sweedler's notation, the above calculations can be expressed more concisely:
we have $\sigma(h\otimes a)=S(h)a_{(-1)}\otimes a_{(0)}$, hence
\begin{multline*}
\sigma\left(a_{(-1)}\otimes a_{(0)}\right)=S\left(a_{(-1)}\right)a_{(0)(-1)}\otimes a_{(0)(0)}=S\left(a_{(-1)(1)}\right)a_{(-1)(2)}\otimes a_{(0)}\\
=\varepsilon\left(a_{(-1)}\right)\otimes a_{(0)}=1\otimes a
\end{multline*}
(the latter equality is the second identity in \eqref{Coaction}, written in Sweedler's notation and taking into account that $\sigma(\delta(a))$ is an element of $H\otimes A\supseteq\K\otimes A=A$) and
\[
\sigma(1\otimes a)=S(1)a_{(-1)}\otimes a_{(0)}=a_{(-1)}\otimes a_{(0)}\;.
\]
\end{proof}

Note that the somewhat involutive identities \eqref{Inv} hold without assuming that $H$ is commutative, nor that $S$ is involutive. In the following proposition we need commutativity of $H$ to get that $S$ and $\mu$, and hence $\sigma$, are algebra homomorphisms.

\begin{prop}\label{Comm}
If $H$ is a commutative Hopf algebra, then every homomor\-phism that satisfies \eqref{Eb1} and \eqref{Eb2} is an isomorphism.
\end{prop}
\begin{proof}
Suppose that $\Theta$ satisfies \eqref{Eb1} and \eqref{Eb2}, let $S$ be the antipode of the Hopf algebra $H$, $\sigma$ as in the statement of \autoref{Isigma} and
\[
\rho:=\sigma^\ast\Theta\;,
\]
which makes sense because $H$ is commutative. Because of \eqref{Inv}, we can assume
\[
\sigma^\ast\left(\eta\otimes\id_A\right)^\ast M=\delta^\ast M\qquad\text{and}\qquad\sigma^\ast\delta^\ast M=\left(\eta\otimes\id_A\right)^\ast M\;,
\]
hence $\rho$ is a homomorphism $\delta^\ast M\to\left(\eta\otimes\id_A\right)^\ast M\overset{\eqref{HM}}{=}H\otimes M$.

Let $\tau:=\left(\mu\otimes\id_A\right)\circ\left(S\otimes\id_{H\otimes A}\right):H\otimes H\otimes A\to H\otimes A$, which is an algebra homomorphism because $H$ is commutative. Then $\Theta$ fulfills $\eqref{Eb1}$ (under our standing assumptions). Let us extend scalars on both sides of that equation via $\tau$. The left-hand side becomes
\[
\left(\tau^\ast\left(\id_H\otimes\delta\right)^\ast\Theta\right)\,\circ\,\left(\tau^\ast\left(\eta\otimes\id_{H\otimes A}\right)^\ast\Theta\right)\;.
\]
Notice that $\tau\circ\left(\id_H\otimes\delta\right)=\sigma$, so that we can assume
\[
\tau^\ast\left(\id_H\otimes\delta\right)^\ast\Theta=\sigma^\ast\Theta=\rho\;,
\]
and $\mu\circ\left(S\otimes\id_H\right)\circ\left(\eta\otimes\id_H\right)=\id_H$, hence $\tau\circ\left(\eta\otimes\id_{H\otimes A}\right)=\id_{H\otimes A}$, so that we can assume
\[
\tau^\ast\left(\eta\otimes\id_{H\otimes A}\right)^\ast\Theta=\Theta\;.
\]
Therefore we have $\rho\circ\Theta$ on the left-hand side. On the right-hand side, since $\tau\circ\left(\Delta\otimes\id_A\right)=\left(\eta\circ\epsilon\right)\otimes\id_A$, we can assume
\[
\tau^\ast\left(\Delta\otimes\id_A\right)^\ast\Theta=\left(\eta\otimes\id_A\right)^\ast\left(\epsilon\otimes\id_A\right)^\ast\Theta\overset{\eqref{Eb2}}{=}\id_{H\otimes M}\;.
\]
Hence $\rho\circ\Theta=\id_{H\otimes M}$.

Extending scalars via $\sigma$ on both sides of the identity $\rho\circ\Theta=\id_{H\otimes M}$, we have
\[
\left(\sigma^\ast\rho\right)\circ\rho=\id_{\delta^\ast\!M}\;.
\]
Therefore $\rho$ has both a right and a left inverse. This suffices to show that $\rho$, and hence $\Theta$, are invertible (\footnote{We also have $\sigma^\ast\rho=\Theta$. This identity can alternatively be obtained in a direct way, by checking that $\sigma\circ\sigma=\id_{H\otimes A}$ (for a commutative Hopf algebra $H$).}).
\end{proof}

Now we show, similarly, that for a commutative Hopf algebra $H$, a left-left relative $\left(A,H\right)$--Hopf module over $\delta$ always corresponds (via $\delta$) to an isomorphism.

\begin{prop}\label{CommHopf}
Suppose that an $(H\otimes A)$--module homomorphism
\[
\rho:\delta^\ast M\to\left(\eta\otimes\id_A\right)^\ast M
\]
and a coaction
\[
\overline\delta:M\to H\otimes M\overset{\eqref{HM}}{=}\left(\eta\otimes\id_A\right)^\ast M
\]
over $\delta$ correspond to each other via $\delta$. If $H$ is a commutative Hopf algebra then $\rho$ is an isomorphism.
\end{prop}

\begin{proof}
Let $S$ be the antipode of $H$, $\sigma$ as in the statement of \autoref{Isigma} and
\[
\tau:=\left(\mu\otimes\id_A\right)\circ\left(S\otimes\id_{H\otimes A}\right)\;.
\]
Since $H$ is commutative, $\sigma$ is an algebra homomorphism, and because of \eqref{Inv} we can assume
\[
\sigma^\ast\left(\eta\otimes\id_A\right)^\ast M=\delta^\ast M\qquad\text{and}\qquad\sigma^\ast\delta^\ast M=\left(\eta\otimes\id_A\right)^\ast M\;.
\]
Therefore we can consider the homomorphism $\sigma^\ast\rho:H\otimes M=\left(\eta\otimes\id_A\right)^\ast M\to\delta^\ast M$, which henceforth can be denoted by $\Theta$. Let us make the assumptions of \autoref{Rho1}. Therefore, since $\overline{\delta}$ is a coaction we have
\[
\left(\Delta\otimes\id_A\right)^\ast\rho=\left(\eta\otimes\id_{H\otimes A}\right)^\ast\rho\circ\left(\id_H\otimes\delta\right)^\ast\rho\;.
\]
Le us extend scalars via $\tau$ on both sides of the above identity. On the left-hand side, since $\tau\circ\left(\Delta\otimes\id_A\right)=\left(\eta\circ\epsilon\right)\otimes\id_A$, we can assume
\[
\tau^\ast\left(\Delta\otimes\id_A\right)^\ast\rho=\left(\eta\otimes\id_A\right)^\ast\left(\epsilon\otimes\id_A\right)^\ast\rho\;.
\]
Since $\overline\delta$ is a coaction, under the assumptions \eqref{C4}, \eqref{Cnew} and according to \autoref{Rho2}, the above homomorphism equals $\id_{H\otimes M}$.

Regarding the right-hand side, we have $\mu\circ\left(S\otimes\id_H\right)\circ\left(\eta\otimes\id_H\right)=\id_H$, hence $\tau\circ\left(\eta\otimes\id_{H\otimes A}\right)=\id_{H\otimes A}$, so that we can assume
\[
\tau^\ast\left(\eta\otimes\id_{H\otimes A}\right)^\ast\rho=\rho\;;
\]
moreover, $\tau\circ\left(\id_H\otimes\delta\right)=\sigma$, so that we can assume
\[
\tau^\ast\left(\id_H\otimes\delta\right)^\ast\rho=\sigma^\ast\rho=\Theta\;.
\]
This way we get $\rho\circ\Theta$. Hence $\rho$ is a left inverse of $\Theta$, and extending scalars via $\sigma$ we also get $\Theta\circ\sigma^\ast\Theta=\id_{\delta^\ast M}$. Therefore $\Theta$ has a right inverse as well, thus it is invertible. Hence $\rho$ is invertible.
\end{proof}

Taking into account the outcome of the preceding subsection, we conclude that for commutative Hopf algebras, to give a homomorphism that satisfies \eqref{SwEb1} and \eqref{SwEb2} is the same as to give a left-left relative $\left(A,H\right)$--Hopf module over $\delta$.

\subsection{Exotic Examples}
The following example shows that a homomorphism $\Theta$ that satisfies \eqref{Eb1} and \eqref{Eb2}, may be not invertible when $H$ is not commutative.

\begin{exm}\label{Ex}
Let $\C^{2\times 2}$ be the $2\times 2$ matrix algebra over the complex numbers, and let us fix $H$ as the free Hopf algebra generated by the dual coalgebra $\left(\C^{2\times 2}\right)^\ast$. The definition is given in \cite[Def.~2]{T1} (cf.\ also \cite[Th.~11]{T1}); but also note that $H$ can be concisely characterized (up to isomorphisms) as being universal among Hopf algebras with a coalgebra morphism of $\left(\C^{2\times 2}\right)^\ast$ into them (cf.\ \cite[Lemma~1]{T1}). As usual, let $S$ denote the antipode. Notice also that $H$ is generated as a $\C$-vector space by $1$ and all elements $S^n\!\left(a^i_j\right)$, with $n$ running on nonnegative integers, $i,j$ running in $\{0,1\}$, and where $a^i_j\in\left(\C^{2\times 2}\right)^\ast$ are the (images in $H$ of the) matrix coefficients.
 
Now, let $A:=H$, $\delta:=\Delta$, $M:=A\oplus A=H\oplus H$ and
\[
\theta:M\to\delta^\ast M=\left(H\otimes H\right)_\Delta\otimes_HM
\]
be the homomorphism over $\eta\otimes\id_A$ defined by
\[
\theta\left(k^0,k^1\right):=\left(a^0_0\otimes k^0+a^0_1\otimes k^1\right)\otimes(1,0)+\left(a^1_0\otimes k^0+a^1_1\otimes k^1\right)\otimes(0,1)\;.
\]
Let
\[
\Theta: H\otimes M\to\left(H\otimes H\right)_\Delta\otimes_HM
\]
be the $(H\otimes H)$--module homomorphism corresponding to $\theta$ via $\eta\otimes\id_A$, that is, $\Theta$ is determined by the condition $\theta(m)=\Theta(1\otimes m)$. It is not difficult (though perhaps a bit cumbersome) to check the conditions \eqref{SwEb1}, \eqref{SwEb2}.

According to \cite[Prop.~4 and Rem.~13]{T1}, there exists a nonzero algebra $R$ and a $\C$-algebra homomorphism $w:H\to R$, such that
\[
\left(
\begin{array}{cc}
w\left(a^0_0\right)&w\left(a^0_1\right)\\\\
w\left(a^1_0\right)&w\left(a^1_1\right)
\end{array}
\right)
=
\left(
\begin{array}{cc}
1&y\\\\
z&yz
\end{array}
\right)
\]
for some appropriate $y,z\in R$.
Let
\[
\Theta_R: R\oplus R\to R\oplus R
\]
be obtained from $\Theta$ by extension of scalars via $w\otimes\varepsilon:H\otimes H\to R$, under the assumptions
\[
R\otimes_{H\otimes H}\left(H\otimes M\right)\ni r\otimes h\otimes\left(k^0,k^1\right)=\left(\varepsilon\!\left(k^0\right)rw(h),\varepsilon\!\left(k^1\right)rw(h)\right)\in R\oplus R\;,
\]
and (using Sweedler notation $\Delta(h)=:h_{(1)}\otimes h_{(2)}$)
\begin{multline*}
R\otimes_{H\otimes H}\left(H\otimes H\right)_\Delta\otimes_HM\ni r\otimes\left(h^0\otimes h^1\right)\otimes\left(k^0,k^1\right)\\
=\left(\varepsilon\!\left(h^1k^0_{(2)}\right)rw\!\left(h^0k^0_{(1)}\right), \varepsilon\!\left(h^1k^1_{(2)}\right)rw\!\left(h^0k^1_{(1)}\right)\right)\in R\oplus R\;.
\end{multline*}
Explicitly,
\[
\Theta_R\!\left(x^0,x^1\right)=\left(x^0+x^1y,x^0z+x^1yz\right)\;.
\]
In particular, $\Theta_R(-y,1)=(0,0)$. Since $(-y,1)\ne (0,0)$, $\Theta_R$ is not an isomorphism. Hence $\Theta$ is not an isomorphism.
\end{exm}

\begin{rem}
The homomorphism in \autoref{Ex} is somewhat induced by the usual left action of $\C^{2\times 2}$ on $\C^2$. If one considers the right action that comes from multiplying row vectors by a matrix, the homomorphisms $\Theta_R$ becomes an isomorphism. But in this case, by considering right modules instead, one gets \[\left(x_0,x_1\right)\mapsto\left(x_0+zx_1,yx_0+yzx_1\right)\;,\]
which has $(-z,1)$ in its kernel. This way one gets an example that works in the right-right case (which is considered in \cite{D}).

Unfortunately, we do not see how the technique of \autoref{Ex} could be adapted to the left-right case, which is considered in \cite{S}, \cite{M}, \cite{Sc}.
\end{rem}

In the next example, the same trick as above gives a \emph{left-right} relative $\left(H, H\right)$--Hopf module with a corresponding $(H\otimes H)$--module homomorphism $\left(H\otimes H\right)_\Delta\otimes_HM\to H\otimes M$ that is not an isomorphism.

\begin{exm}\label{ExHopf}
In notation of \autoref{Ex}, let us consider the homomorphism
\[
\overline\delta:M\to M\otimes H
\]
over $\delta=\Delta$ defined by
\begin{multline*}
\overline\delta\left(k^0,k^1\right):=
\left(k^0_{(0)},0\right)\otimes k^0_{(1)}a^0_0+\left(k^1_{(0)},0\right)\otimes k^1_{(1)}a^0_1\\
+\left(0,k^0_{(0)}\right)\otimes k^0_{(1)}a^1_0+\left(0,k^1_{(0)}\right)\otimes k^1_{(1)}a^1_1
\end{multline*}
(under sumless Sweedler notation $\delta(h)=:h_{(0)}\otimes h_{(1)}$). A routine verification confirms that $\overline\delta$ is a coaction, and hence defines a left-right relative $\left(H, H\right)$--Hopf module, provided that $\delta=\Delta$ is considered as a right coaction.

If
\[
\rho:\left(H\otimes H\right)_\Delta\otimes_HM\to M\otimes H
\]
is the $(H\otimes H)$--module homomorphism corresponding to $\overline\delta$ via $\delta$, and \[\rho_R:R\oplus R\to R\oplus R\] is obtained from it by extension of scalars via $\varepsilon\otimes w$ (under identifications similar to those in \autoref{Ex}), we have $\rho_R=\Theta_R$ and hence it is not an isomorphism. Therefore $\rho$ is not an isomorphism.
\end{exm}

\section{Noncommutative equivariant bundles and fibered categories}

\subsection{Basic facts}
Fibered categories are a proper context in which natural properties of fiber bundles, modules and sheaves can be analyzed. We assume \cite[Def.~3.5]{V} as the definition of a fibered category over a category $\mathcal{C}$. For the reader convenience, we also rewrite below the definition given in \cite[Def.~3.46]{V} of an equivariant object (see also \cite[4.1]{S}), provided that $G$ is a (covariant) functor $\mathcal{C}^\mathrm{op}\to(\mathrm{Grp})$ into the category of groups (\footnote{In the preamble of \cite[Sect.~3.8]{V} one finds $\mathcal{C}\to(\mathrm{Grp})$, probably by a misprint.}).
\begin{defn}\label{Recap}
Let $F:\mathcal{C}^\mathrm{op}\to(\mathrm{Set})$ and $G:\mathcal{C}^\mathrm{op}\to(\mathrm{Grp})$ be functors.

A \emph{left action} of $G$ on $F$ is a natural transformation
\[
G\times F\to F
\]
such that for every object $U$ of $\mathcal{C}$, we have that
\[
G(U)\times F(U)\to F(U)
\]
is a left action of the group $G(U)$ on the set $F(U)$. A \emph{left action} of $G$ on an object $X\in\mathcal{C}$ is a left action of $G$ on the contravariant hom functor $h_X$, $h_X(Y):=\Hom_{\mathcal{C}}(Y,X)$, $h_X(f)(\varphi):=\varphi\circ f$.

Given such a left action, let $\mathcal{F}$ be a fibered category over $\mathcal{C}$, with defining functor $p_{\mathcal{F}}:\mathcal{F}\to\mathcal{C}$. A \emph{$G$-equivariant object} of $\mathcal{F}$ is an object $\pi$ together with a left action of $G\circ p_{\mathcal{F}}$ on $\pi$, such that for all objects $\xi$ of $\mathcal{F}$ the map \[\Hom_{\mathcal{F}}(\xi,\pi)\to\Hom_{\mathcal{C}}\left(p_{\mathcal{F}}(\xi),p_{\mathcal{F}}(\pi)\right)\] given by the functor $p_{\mathcal{F}}$ is a $G\left(p_{\mathcal{F}}(\xi)\right)$--map (that is, it is equivariant with respect to the action of $G\left(p_{\mathcal{F}}(\xi)\right)$).
\end{defn}

In the rest of the present work, we let $\mathcal{C}$ be the opposite of the category of (not necessarily commutative) $\K$-algebras. To avoid confusion, given $\K$-algebra homomorphisms $\varphi:A\to B$, $\psi:B\to C$ (hence, arrows $B\to A$ and $C\to B$ in $\mathcal{C}$), we shall follow the common usage to denote composition: $\psi\circ\varphi$ (that is, we shall not use the symbol $\circ$ to denote the composition in $\mathcal{C}$). The following action of functors will be considered.

\begin{exm}\label{BA}
For our $\K$-algebra $A$ and $\K$-bialgebra $H$ (see \autoref{Standing}), and for each $\K$-algebra $B$, let us consider the $\K$-vector spaces $\Hom_\K(A,B)$ and $\Hom_\K(H,B)$ (sets of $\K$-vector spaces homomorphism). The convolution product
\[
f\ast g:=\mu_B\circ\left(f\otimes g\right)\circ\Delta\;,
\]
with $\mu_B:B\otimes B\to B$ being the multiplication of $B$, allows to define on $\Hom_\K(H,B)$ a (unital) $\K$-algebra structure. In particular we have a monoid with convolution as the (only) operation, which we denote by $Z(B)$. The subset of invertible elements is a group, which we denote by $G(B)$. This way we get functors $Z:\mathcal{C}^\mathrm{op}\to (\mathrm{Mon})$, $G: \mathcal{C}^\mathrm{op}\to (\mathrm{Grp})$.

If $f:H\to B$, $\varphi:A\to B$ are homomorphisms of $\K$-vector spaces, then
\[
f\bullet\varphi:=\mu_B\circ\left(f\otimes\varphi\right)\circ\delta\in\Hom_\K(A,B)\;.
\]
This gives a left monoid action
\[
Z(B)\times\Hom_\K(A,B)\to\Hom_\K(A,B)\;,\quad (f,\varphi)\mapsto f\bullet\varphi\;,
\]
and, by restriction, a left (group) action
 \[
G(B)\times\Hom_\K(A,B)\to\Hom_\K(A,B)\;.
\]
Denoting by $F$ the functor $\mathcal{C}^\mathrm{op}\to(\mathrm{Set})$ given by $F(B):=\Hom_\K(A,B)$, $F(g)(\varphi):=g\circ\varphi$, $g\in\Hom_{\mathcal{C}}(B',B)$, $\varphi\in F(B)$, we get natural transformations
\[
Z\times F\to F\;,\qquad G\times F\to F\;.
\]
The latter is a left action according to \autoref{Recap}, and the former can be regarded as a left action in a slightly extended sense.
\end{exm}

\begin{rem}\label{Fa}
In notation of the above example, when $B$ is commutative, that is, when $\mu_B$ is a $\K$-algebra homomorphism, $f\ast g$ and $f\bullet\varphi$ are $\K$-algebra homomorphisms. When, in addition, $H$ is a Hopf algebra with antipode $S$, we have that every $\K$-algebra homomorphism $f:H\to B$ is invertible in $F(B)$, with $f\circ S$ as its inverse. Hence, in this case $\Hom_{\mathcal{C}}(B,H)$ is a subgroup of $G(B)$, and $\Hom_{\mathcal{C}}(B,A)\subseteq F(B)=\Hom_\K(A,B)$ is invariant under the action $\Hom_{\mathcal{C}}(B,H)\times\Hom_\K(A,B)\to\Hom_\K(A,B)$.

This way, we get a left action
\[
h_H\times h_A\to h_A
\]
on the subcategory $\mathcal{C}_{\mathrm{ab}}$ of $\mathcal{C}$ given by commutative $\K$-algebras (where $h_H(B):=\Hom_{\mathcal{C}}(B,H)$, $h_A(B)=\Hom_{\mathcal{C}}(B,A)$). When the Hopf algebra $H$ is commutative as well (hence a group object in $\mathcal{C}_{\mathrm{ab}}$), $h_H$ is its (contravariant) hom-functor in $\mathcal{C}_{\mathrm{ab}}$. When, in addition, $A$ is commutative, $\delta$ is an action in $\mathcal{C}_{\mathrm{ab}}$, and we recover the action $h_H\times h_A\to h_A$ induced by $\delta$ in the usual way (see, e.g., \cite[Prop.~2.16]{V}).

Unfortunately, in spite of the fact that the action has been defined for general $A$ and (Hopf) $H$, this gives little additional information in the noncommutative case. Indeed, $h_H\times h_A\to h_A$ is isomorphic to the action $h_{H_{\mathrm{ab}}}\times h_{A_{\mathrm{ab}}}\to h_{A_{\mathrm{ab}}}$ given by the abelianization  $\delta_{\mathrm{ab}}:A_{\mathrm{ab}}\to\left(H\otimes A\right)_{\mathrm{ab}}=A_{\mathrm{ab}}\otimes H_{\mathrm{ab}}$ of $\delta$.
\end{rem}

\subsection{Fibered categories of left modules}

The class $\mathcal{F}$ of all (small) left modules over $\K$-algebras can be considered as a fibered category over $\mathcal{C}$ in various ways. To this end, let $A,B\in\mathcal{C}$, $M$ a left $A$-module, $N$ a left $B$-module, $\varphi:A\to B$ a $\K$-algebra homomorphism, that is, a morphism $B\to A$ in $\mathcal{C}$. The perhaps first option to define a morphism $N\to M$ in $\mathcal{F}$ is to take a pair $(f,\varphi)$ such that $f:M\to N$ is a module homomorphism over $\varphi$ (this means as before that $f$ is additive ad $f(am)=\varphi(a)f(m)$). Note that $f$ can also be described as an $A$-module homomorphism of $M$ into the module $\varphi_\ast N$ obtained from $N$ by restriction of scalars, and it corresponds to a $B$-module homomorphism $\varphi^\ast M\to N$. In this case, to fix a cleavage of $\mathcal{F}$, that is, a choice of a cartesian morphism $N\to M$ in $\mathcal{F}$ for every given left $A$-module $M$ and $\K$-algebra homomorphism $\varphi$, is the same as to fix an extension of scalars $M\to\varphi^\ast M$, still for every $M$ and $\varphi$ (usually such a choice is given by some general construction for tensor products). In this framework, we can say that the conditions \eqref{Eb1}, \eqref{Eb2} that we considered on $\Theta:\left(\eta\otimes\id_A\right)^\ast M\to\delta^\ast M$ in the previous section, are written under a choice of a cleavage that fulfills \eqref{C1}, \eqref{C2}, \eqref{C3} and $\eqref{C4}$ (\footnote{It can be proven that such a cleavage exists for each given $\Theta$ (but this does not mean that there exists a single cleavage that works, in that respect, for every $\Theta$; note also that by carefully checking existence for each $\Theta$, one discovers that when $\delta=\eta\otimes\id_A$ we have that \eqref{C3} is a consequence of \eqref{C1} and \eqref{C2}). That is why in \autoref{Disc} we said that our conventions on tensor products and extension of scalars consists in that a cleavage is chosen in the course of the exposition. In more complicated situations such a context-depending cleavage may well not exist, but actually it is not strictly needed: when two different cartesian morphisms with the same target and base arise, one may denote them differently (or abuse the notation).}).

Another option is given again by a pair $(f,\varphi)$, where now $f$ is a $B$-module homomorphism $N\to\varphi^\ast M$, with $\varphi^\ast M$ taken from some fixed cleavage of the former fibered category. As a third option, $f$ in the pair $(f,\varphi)$ may be taken as an $A$-module homomorphism $\varphi_\ast N\to M$, which corresponds to a $B$-module homomorphism of $N$ into the module $\Hom_A(B,M)$ obtained from $M$ by coextension of scalars.

As extensively explained in \autoref{Disc}, in this work we choose the middle option because in geometric situations it corresponds to the familiar notion of a vector bundle morphism over different bases. Moreover, with that choice, in the commutative case we have a natural embedding into the category of quasi-coherent sheaves on schemes. On the other hand, we do not intend to prevent other choices which may turn to be interesting as well. Hence, from now on $\mathcal{F}$ will denote the category of left modules over $\K$-algebras with morphisms $(f,\varphi)$ with $f$ being a $B$-module homomorphism $N\to\varphi^\ast M$. We shall sometimes refer to the cleavage given by the choices of $M\to\varphi^\ast M$ as the \emph{defining cleavage} (which, strictly speaking, is \emph{not} a cleavage of $\mathcal{F}$). The composition $(f,\varphi)\circ (g,\psi)$ is the morphism \[\left(\psi^\ast f\circ g,\psi\circ\varphi\right)\] when $M$ is such that $\psi^\ast\varphi^\ast M=(\psi\circ\varphi)^\ast M$ (and where $\circ$ in $\psi^\ast f\circ g$ and $\psi\circ\varphi$ denotes the compositions of maps of modules and maps of algebras). In general, one has $(f,\varphi)\circ (g,\psi):=\left(\iota\circ\psi^\ast f\circ g,\psi\circ\varphi\right)$,
with $\iota:\psi^\ast\varphi^\ast M\to(\psi\circ\varphi)^\ast M$ being the natural isomorphism. A cleavage of $\mathcal{F}$ is clearly given by the morphisms $\left(\id_{\varphi^\ast M},\varphi\right):\varphi^\ast M\to M$ (and clearly depends on the choice of the defining cleavage).

\begin{rem}
In our standing notation (see \autoref{Standing}), we have that $\left(\Theta,\delta\right)$ is a morphism $\left(\eta\otimes\id_A\right)^\ast M\to M$ in $\mathcal{F}$, under the assumption $H\otimes M=\left(\eta\otimes\id_A\right)^\ast M$.  We can equivalently express the conditions \eqref{Eb1} and \eqref{Eb2} in terms of morphism identities in $\mathcal{F}$, more similar to the coaction conditions \eqref{Coaction}. Indeed, let us suppose that the defining cleavage fulfills \eqref{C1}, \eqref{C2} and $\eqref{C4}$. Because of \eqref{C1}, $\left(\,\left(\eta\otimes\id_{H\otimes A}\right)^\ast\Theta\,,\,\id_H\otimes\delta\,\right)$ is a morphism into $\left(\eta\otimes\id_A\right)^\ast M$. Because of \eqref{C2}, we have that \eqref{Eb1} is equivalent to
\[
\left(\Theta,\delta\right)\circ\left(\,\left(\eta\otimes\id_{H\otimes A}\right)^\ast\Theta\,,\,\id_H\otimes\delta\,\right)=\left(\Theta,\delta\right)\circ\left(\,\id_{\left(\Delta\otimes\id_A\right)^\ast\left(\eta\otimes\id_A\right)^\ast M}\,,\,\Delta\otimes\id_A\,\right)\;.
\]
Because of \eqref{C4}, $\left(\id_M,\varepsilon\otimes\id_A\right)$ is a morphism into $\left(\eta\otimes\id_A\right)^\ast M$, and we have that \eqref{Eb2} is equivalent to
\[
\left(\Theta,\delta\right)\circ\left(\,\id_M\,,\,\varepsilon\otimes\id_A\,\right)=\left(\id_M,\id_A\right)\;.
\]
\end{rem}

To properly work in general (regardless of the assumptions), the above conditions should to be intended up to replacing $\mathcal{F}$ with an isomorphic category obtained by changing the defining cleavage; alternatively, some natural isomorphisms between modules in \eqref{C1}, \eqref{C2} and $\eqref{C4}$, respectively, should be explicitly displayed.

\subsection{The commutative case}
\label{OkEq}

In the above settings, our motivations for considering homomorphisms that satisfy \eqref{SwEb1} and \eqref{SwEb2} (or \eqref{Eb1} and \eqref{Eb2}) as noncommutative generalizations of equivariant bundles can be summarized as follows.

Since a commutative Hopf algebra is a group object in the category $\mathcal{C}_{\mathrm{ab}}$ of commutative $\K$-algebras, \cite[Prop.~3.47]{V} shows that  when $A$ is commutative as well, to give an equivariant object in $\mathcal{C}_{\mathrm{ab}}$ is the same as to give a homorphism $\Theta$ that satisfies \eqref{Eb1} and \eqref{Eb2} (and hence a homomorphism $\theta$ that satisfies \eqref{SwEb1} and \eqref{SwEb2}). Moreover, in this case \autoref{Comm} is a consequence of \cite[Prop.~3.49]{V}. 

If one restricts $\mathcal{C}$ to algebras of regular functions on affine algebraic varieties, and $\mathcal{F}$ to finitely generated and locally free modules over them, we have an equivalence with the category of algebraic vector bundles over affine varieties. Equivariant objects obviously correspond to usual equivariant (algebraic) vector bundles.

\subsection{Obstacles in the noncommutative case}

Unfortunately, to fit noncommutative Hopf algebras in the above framework is problematic. The basic reason is that they are not group objects in $\mathcal{C}$ (for further discussions see \cite[nn.~3.1, 3.2]{S}). To overcome this difficulty, one might attempt, for instance, to consider the action $G\times F\to F$ of \autoref{BA}. Let us briefly pursue this idea.

In order to define an action of $G(N)$ on $\Hom_{\mathcal F}(N,M)$ for each $N\in\mathcal{F}$, as required in \autoref{Recap}, let us consider a morphism $(f,\varphi):N\to M$ in $\mathcal{F}$ over a morphism $B\to A$ in $\mathcal{C}$, that is, $\varphi:A\to B$ is a $\K$-algebra homomorphism and $f:N\to\varphi^\ast M$ a $B$-module homomorphism. Let $v:H\to B$ be a homomorphism of $\K$-vector spaces and suppose that $\psi:=\mu_B\circ(v\otimes\varphi)=H\otimes A\to B$ is a $\K$-algebra homomorphism (as it is always the case when $B$ is commutative and $v$ is a $\K$-algebra homomorphism). Since $\psi\circ\left(\eta\otimes\id_A\right)=\varphi$, there exists a unique morphism $\left(g,\psi\right):N\to\left(\eta\otimes\id_A\right)^\ast M$ over $\psi$ such that \[\left(\id_{\left(\eta\otimes\id_A\right)^\ast M},\eta\otimes\id_A\right)\circ\left(g,\psi\right)=(f,\varphi)\;.\]
We have that $(\Theta,\delta)\circ\left(g,\psi\right)$ is a morphism $N\to M$ over $\psi\circ\delta=\mu_B\circ(v\otimes\varphi)\circ\delta$, that is, in notation of \autoref{BA}, over $v\bullet\varphi$. The above construction allows every $v$ for which $\psi$ is a $\K$-algebra homomorphism to act on $\Hom_{\mathcal F}(N,M)$. In the commutative case, this defines $M$ as an equivariant object by means of $\left(\Theta,\delta\right)$ as in \cite[Prop.~3.47]{V}.

For arbitrary $H$ and $A$, to have something meaningful (hopefully, to define $M$ as an equivariant object in such a way that $\left(\Theta,\delta\right)$ can be reconstructed back) we would need a set of allowable $v$ as large as possible. Unfortunately, the condition on $\psi$ of being a $\K$-algebra homomorphism is quite restrictive. For instance notice that, since $\psi\circ\left(\id_H\otimes\eta_A\right)=v$, with $\eta_A$ being the unit map of $A$ and assuming $\id_H\otimes\K=H$, we have that $v$ must be a $\K$-algebra homomorphism. Besides, to have a functor into groups we also need $v$ to be invertible in the convolution monoid $\Hom_\K(H,B)$ (cf.~\autoref{Fa}).

One might try to consider only commutative $\K$-algebras $B$, but this would be too restrictive. Indeed, as noted at the end of \autoref{Fa}, in this case the action on the base is isomorphic to the action given by the abelianizations of $H$ and $A$; and it is not difficult to see that the action on $\Hom_{\mathcal F}(N,M)$ given by $\Theta$ is isomorphic to the action obtained by replacing $M$ with $M_{\mathrm{ab}}:=A_{\mathrm{ab}}\otimes_AM$ and changing $\Theta$ accordingly.

\subsection{Noncommutative equivariant bundles and monads}

Let $_A\mathcal M$ be the category of left $A$-modules. In \cite[4.4]{S} it is pointed out that if we give $H\otimes M$ the left $A$-module structure induced by $\delta$, we obtain an endofunctor of $_A\mathcal M$, which is a comonad $\mathbf{G}$ in a natural way. Moreover, like in \cite[4.4.2]{S}, one easily recognizes that left-left relative $(A,H)$--Hopf modules are comodules over the comonad $\mathbf{G}$. To this end, let us mention that by a (left) comodule over a comonad $T:\mathcal{C}\to\mathcal{C}$ may be meant a functor $F:\mathcal{C}'\to\mathcal{C}$, together with a natural transformation $F\to TF$ that satisfies two natural compatibility conditions with the structure of $T$. However, a more restricted meaning is often in use: by a comodule over $T$ (also called a coalgebra over $T$) is simply meant an object $F$ of $\mathcal{C}$, together with a morphism $f:F\to TF$ such that $\varepsilon_TF\circ f=\id_F$ and $\delta_TF\circ f=Tf\circ f$, with $\varepsilon_T$ and $\delta_T$ being the structural natural transformations of the comonad $T$. One can regard the latter notion as a particular case of the former, simply by replacing $F$ with the functor of the (terminal) category with one morphism $\iota$, into the category $\mathcal{C}$, that sends $\iota$ into $\id_F$. It is easy to check that with the restricted notion, $\mathbf{G}$-comodules are precisely left-left relative $(A,H)$--Hopf modules.

Our purpose in this concluding subsection is to make a similar construction for homomorphisms that satisify \eqref{SwEb1} and \eqref{SwEb2}. To this end, we consider the endofunctor $\mathbf H$ of $_A\mathcal M$ that associates with each left $A$-module $M$ the codomain $\delta^\ast M$ of (all) homomorphisms $\theta:M\to\delta^\ast M$ (and $\Theta$), considered as an $A$-module by restriction of scalars via $\nu:=\eta\otimes\id_A:A\to H\otimes A$ (the action on homomorphisms is obviously $f\mapsto\delta^\ast f$). Concisely: $\mathbf{H}=\nu_\ast\delta^\ast$.

Let $_\nu(H\otimes A)$ denote $H\otimes A$ equipped with the left $A$-module structure induced by $\nu$ and $(H\otimes A)_\varphi$ be $H\otimes A$ with the right $A$-module structure induced by a $\K$-algebra homomorphism $\varphi:A\to H\otimes A$. Since for every right $A$-module $N$, the vector space $H\otimes N$ is a tensor product $N\otimes_A{}_\nu(H\otimes A)$ (with $n\otimes(h\otimes a)=h\otimes na$), we have that $H\otimes H\otimes A$ is a tensor product $(H\otimes A)_\varphi\otimes_A{}_\nu(H\otimes A)$. Hence, we can assume
\[
(H\otimes A)_\varphi\otimes_A{}_\nu(H\otimes A)=H\otimes H\otimes A
\]
with, in a sumless Sweedler notation for $\varphi$,
\[
\left(h\otimes a\right)\otimes\left(h'\otimes a'\right)=h'\otimes ha'_{(-1)}\otimes aa'_{(0)}\;.
\]
Let $_\nu(H\otimes A)_\varphi$ be the $A$-bimodule with the left structure induced by $\nu$ and the right structure induced by $\varphi$. We have an $A$-bimodule \[_\nu(H\otimes H\otimes A)_\varphi={_\nu(H\otimes A)_\varphi\otimes_A{}_\nu(H\otimes A)_\varphi}\] such that
\[
a_{\mathrm l}\left(h\otimes h'\otimes a\right)a_{\mathrm r}=h{a_{\mathrm r}}_{(-1)}\otimes h'{a_{\mathrm r}}_{(0)(-1)}\otimes a_{\mathrm l}a{a_{\mathrm r}}_{(0)(0)}\;.
\]
One easily checks that $\Delta\otimes\id_A:{}_\nu(H\otimes A)_\varphi\to{}_\nu(H\otimes H\otimes A)_\varphi$ is a left module homomorphism, and that it is a right module homomorphism if and only if $(\id_H\otimes\varphi)\circ\varphi=(\Delta\otimes\id_A)\circ\varphi$. Since $\mathbf H$ is the tensor product by $_\nu(H\otimes A)_\delta$ on the left, when $\varphi$ is the coaction $\delta$ we have that $\Delta\otimes\id_A$ gives a natural transformation $\mathbf H\to\mathbf H\mathbf H$. Moreover, the homomorphism $\varepsilon\otimes\id_A$ gives a natural transformation of $\mathbf H$ into the identity functor. Taking into account the coalgebra properties of $\Delta$ and $\varepsilon$, one easily checks that in this way $\mathbf H$ becomes a comonad.

To show that $\theta$ makes $M$ an $\mathbf H$-comodule (in the restricted sense) if and only if satisfies \eqref{SwEb1} and \eqref{SwEb2}, it suffices another, not difficult check.

\end{document}